\documentclass[12pt]{amsart}
\usepackage{geometry}                
\usepackage{graphicx}
\usepackage{amssymb}
\usepackage{epstopdf}
\usepackage{color}
\usepackage{hyperref}
\usepackage{animate}
\usepackage{enumitem}

\newtheorem{theorem}{Theorem}[section]

\newtheorem{corollary}[theorem]{Corollary}

\title{On the maximal mean curvature of a smooth surface}
\author[V. Ferone, C. Nitsch, C. Trombetti]
       {Vincenzo Ferone, Carlo Nitsch, Cristina Trombetti}
       \address[V.~Ferone]{Universit\`a degli Studi di Napoli Federico II, Italy.}
\email{ferone@unina.it}
\address[C.~Nitsch]{Universit\`a degli Studi di Napoli Federico II, Italy.}
\email{c.nitsch@unina.it}
\address[C.~Trombetti]{Universit\`a degli Studi di Napoli Federico II, Italy.}
\email{cristina@unina.it}


\begin{document}
\maketitle

\begin{abstract}
Given a smooth simply connected planar domain, the area is bounded away from zero in terms of the maximal curvature alone. We show that in higher dimensions this is not true, and for a given maximal mean curvature we provide smooth embeddings of the ball with arbitrary small volume.
\end{abstract}

\section{Introduction}
According to a classical result \cite{BZ,PI}, for any smooth simple planar curve $\gamma$, if the curvature $\kappa$ is bounded from above by some positive constant $M$, then the curve $\gamma$ encloses a bounded simply connected domain $\Omega$ which contains a disk of radius $M^{-1}$. In particular, see \cite{P},  if $\Omega^*$ is a disk having same measure as $\Omega$, the following inequality holds true 
\begin{equation}\label{eq_prima}
\|\kappa\|_{L^\infty(\partial\Omega)} \ge \|\kappa\|_{L^\infty(\partial\Omega^*)}
\end{equation}
or equivalently
\[
\lVert\kappa\rVert_{L^\infty(\partial\Omega)}^2 Area(\Omega)\ge \pi,
\]
equality holding in both cases if and only if $\Omega$ is a disk.

Very recently such inequalities have been generalized to other $L^p$ norm of the curvature. More precisely for $p=2$ (see \cite{BH,FKN,G}) and for $p\ge 1$ (see \cite{FKN2}) it holds
\[
\|\kappa\|_{L^p(\partial\Omega)} \ge \|\kappa\|_{L^p(\partial\Omega^*)}
\]
or equivalently
\[
\lVert\kappa\rVert_{L^p(\partial\Omega)}^2 Area(\Omega)\ge 2^{2/p}\pi^{(p+1)/p},
\]
equality holding again in both cases if and only if $\Omega$ is a disk.

In this short note we consider the 3-dimensional case. We replace simply connected planar domain by sets diffeomorphic to balls in $\mathbb{R}^3$, and the planar curvature $\kappa$ by the mean curvature $H$ of the boundary. If by $S(\cdot)$ 
and $V(\cdot)$ we denote  surface area and volume respectively, our main result then reads as follows.

\begin{theorem}\label{main}
For any $\epsilon>0$ there exists $\Omega_\epsilon\subset \mathbb{R}^3$, diffeomorphic to the unit ball, with smooth $C^{1,1}$ boundary, such that 
\[
\|H\|_{L^\infty(\partial \Omega_\epsilon)}\le 1, \quad |S(\Omega_\epsilon)-8\pi| \le\epsilon,\quad \mbox{and} \quad V(\Omega_\epsilon) \le \epsilon.\]
\end{theorem}

If $\Omega^*$ denotes the ball having same volume as $\Omega$, a first consequence is
\begin{corollary}\label{corollario}
In the class of $C^{1,1}$ subsets of $\mathbb{R}^3$ diffeomorphic to balls for all $2<p\le\infty$ we have ù
\begin{equation}\label{eq_inf}
\inf\left\{\left\|H\right\|_{L^p(\partial\Omega)}: V(\Omega)=1\right\}=0.
\end{equation}
In particular, for all $2<p\le\infty$, there exists a $C^{1,1}$ set $\Omega\subset\mathbb{R}^3$, diffeomorphic to the ball $\Omega^*$, such that 
\[
\|H\|_{L^p(\partial\Omega)} < \|H\|_{L^p(\partial\Omega^*)}.
\]
\end{corollary}
Notice that $\|H\|_{L^2}$ is a very special case since it corresponds to the Willmore energy (invariant under dilation) which is indeed minimal on balls \cite{W}. The case $p=2$ is also a threshold case since $\|H\|_{L^p}$, for $p<2$, scales under dilation as a positive power of the volume. Under volume constraint $\|H\|_{L^p}$ is in fact bounded away from zero for all $1\le p<2$, even if optimal lower bounds are still unknown, see \cite{HDMT,T}.


Our interest in this kind of inequalities is also due to a question arisen in \cite{KP,PP} in relation to estimates for Laplacian eigenvalue with Robin boundary conditions. 

For any given $\alpha>0$, consider the eigenvalue problem
\begin{equation}\label{eq_robin}
\left\{\begin{array}{l}
-\Delta u=\lambda u\quad \mbox{in } \Omega\\\\
\displaystyle\frac{\partial u}{\partial \nu}=\alpha u \qquad \mbox{on } \partial\Omega
\end{array}
\right.
\end{equation}
By $\lambda(\Omega,\alpha)$ we denote the greatest (negative) $\lambda$ such that \eqref{eq_robin} admits a nontrivial solution, namely:
\[
\lambda(\Omega,\alpha)=\max \left\{ \alpha\displaystyle\int_{\partial\Omega} v^2 -\displaystyle \int_{\Omega}|\nabla v|^2: v\in H^1(\Omega), \displaystyle\int_{\Omega}v^2=1 \right\}
\]

It has been conjectured for long time \cite{B} that balls achieve the greatest eigenvalue among sets of given measure. Indeed they are local maximizers in any dimension (see \cite{FNT}) and global maximizers in $2$ dimensions for $\alpha$ small enough (see \cite{FK}). However in \cite{FK} the authors also show that large values of $\alpha$ provides the annulus as counterexample to the conjecture.

In \cite{KP,PP} this was clarified showing that whenever $\Omega\subset\mathbb{R}^n$ is $C^{1,1}$ then 
\[
\lambda(\Omega,\alpha)=-\alpha^2 -\alpha(n-1)\sup_{\partial\Omega}H+o(\alpha)\qquad \mbox{as $\alpha\to\infty$}
\]

In fact, since any annulus $\mathcal{A}$ having same volume as a ball $B$ also has a smaller maximal curvature, we have $\lambda(\mathcal{A},\alpha)>\lambda(B,\alpha)$ as soon as $\alpha$ is large enough. Thereafter in \cite{PP} the authors were interested in minimizing the maximal curvature in classes of domains of given volume subject to some kind of additional topological constraints. 
In view of \eqref{eq_prima}, in two dimension balls achieve the minimal maximal mean curvature whenever we restrict to simply connected sets. In dimension greater than $2$ they were able as well to prove that starshapedness is enough to get the same result. Whence they left open the following problem.
\vskip .1cm
{\bf Question.} Let $\Omega\in \mathbb{R}^n$ be a bounded smooth domain with a connected boundary and let $\Omega^*$ be a ball with same volume. Do we have $\displaystyle \sup_{\partial\Omega}H\ge \sup_{\partial\Omega^*}H$?
\vskip .1cm
Corollary \ref{corollario} provides a negative answer.






\section{Proofs}

\begin{proof}[Proof of Theorem \ref{main}]
The proof relies on an explicit construction. For all $h,\beta>0$ we design a $C^{1,1}$ arc of curve $\Gamma_{h,\beta}$ by joining together five pieces of arcs $\gamma_1...\gamma_5$ as follows (see Figure \ref{fig_2d}).
The arc $\gamma_1$ is the arc of curve whose parametric representation for $t\in [0,\pi]$ is: 

\begin{equation}
\gamma_1:\left\{
\begin{array}{ll}
x(t) = \displaystyle\frac{1}{2h}\left (\sqrt{\beta + \sin^2 t}  -\sin t  \right) 
\cr 
y(t)= \displaystyle\frac{1}{2h}\left ( \cos t-1 + \displaystyle\int_0^t \frac{\sin^2 s}{\sqrt{\beta + \sin^2 s}} ds \right) 
\end{array}\right.
\end{equation}
In Figure \ref{fig_2d} the arc $\gamma_1$ is represented in magenta.\\
\begin{figure}[htbp]
\begin{center}
\def\svgwidth{3.in}
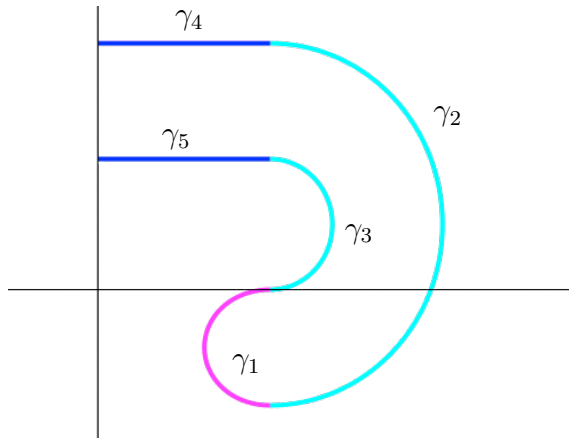\caption{The curve $\Gamma_{h,\beta}$ with $h=1$ and $\beta=4$.}
\label{fig_2d}
\end{center}
\end{figure}
With $\gamma_2$ and $\gamma_3$ we denote two concentric semicircles with parametric representations for $t\in[-\pi/2,\pi/2]$ given by:
\begin{equation}
\gamma_2:\left\{
\begin{array}{ll}
x(t) = \displaystyle\frac{1}{h}\cos t + \frac{\sqrt {\beta}}{2h}\\\\
y(t)= \displaystyle\frac{1}{h}\sin t + R(h,\beta)
\end{array}\right.
\end{equation}
\begin{equation}
\gamma_3:\left\{
\begin{array}{ll}
x(t) = \displaystyle R(h,\beta)\cos t+\frac{\sqrt {\beta}}{2h}\\\\
y(t)= \displaystyle R(h,\beta)\sin t+R(h,\beta)\end{array}\right.
\end{equation}
Here $\displaystyle R(h,\beta)=\frac{1}{2h}\int_0^\pi \frac{\sin^2 s}{\sqrt{\beta + \sin^2 s}}ds<\frac1h$.
In Figure \ref{fig_2d} the arcs $\gamma_2$ and $\gamma_3$ are represented in cyan.

The remaining arcs, $\gamma_4$ and $\gamma_5$, are segments parallel to the $x$-axis, namely for $t\in[0,1]$
\begin{equation}
\gamma_4:\left\{
\begin{array}{ll}
x(t) = \displaystyle t\frac{\sqrt{\beta}}{2h}\\
\cr 
y(t)= \displaystyle R(h,\beta)+\frac1h 
\end{array}\right.
\end{equation}
\begin{equation}
\gamma_5:\left\{
\begin{array}{ll}
x(t) =  \displaystyle t\frac{\sqrt{\beta}}{2h}\\
\cr 
y(t)= \displaystyle 2R(h,\beta)
\end{array}\right.
\end{equation}

In Figure \ref{fig_2d} the arcs $\gamma_4$ and $\gamma_5$ are represented in blue.

We can consider now the surface of revolution $\Sigma_{h,\beta}$ obtained rotating of an angle $2\pi$ the curve $\Gamma_{h,\beta}$ around the $y$-axis. Since $\Gamma_{h,\beta}$ is simple and lies in the half plane $x\ge0$, the surface $\Sigma_{h,\beta}$ is $C^{1,1}$, without self intersection, compact, and it is the boundary of a bounded connected set $\Omega_{h,\beta}$. From now on we choose as normal to $\Sigma_{h,\beta}$ the unit outer normal of $\Omega_{h,\beta}$. One half of $\Sigma_{h,\beta}$ is represented in Figure \ref{fig_3d}.
\begin{figure}[htbp]
\begin{center}
\animategraphics[loop,autoplay,viewport=200 100 600 350, scale=.5]{1}{3d_}{0}{9}
\includegraphics[width=3.in]{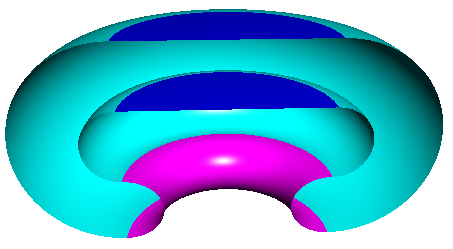}\caption{The surface of revolution obtained by rotating $\Gamma_{h,\beta}$ in Fig. \ref{fig_2d} around the $y$ axis. Here $h=1$ and $\beta=4$.
}
\label{fig_3d}
\end{center}
\end{figure}
 
Let us compute the mean curvature of $\Sigma_{h,\beta}$. Since $\Gamma$ is $C^{1,1}$ then also $\Sigma$ is $C^{1,1}$ with bounded mean curvature $H$ defined $\mathcal{H}^2$ almost everywhere. More specifically, taking into account the orientation choosen above:
\begin{itemize}
\item The arc $\gamma_1$ is called nodary and it generates a well known surface called nodoid and has the characteristic that the mean curvature $H$ is constant and equal to $h$ (see \cite{K}).
\item The arc $\gamma_4$ and $\gamma_5$ generate two flat disks and therefore have vanishing mean curvature $H\equiv 0$.
\item The arc $\gamma_2$ is half circle of radius $1/h$ and generates a portion of torus with mean curvature $0<H\le h\left(1-\frac{\sqrt{\beta}}{4+2\sqrt{\beta}}\right)$.
\item The arc $\gamma_3$ is half circle of radius $R(h,\beta)$ and generates a portion of torus with mean curvature
$0>H\ge-\frac{1}{R} \left(      1-\frac{ \sqrt{\beta} }{4hR+2\sqrt{\beta}}      \right)$ 
\end{itemize}
 
For fixed $h>0$ we vary $\beta$ between $0$ and $1$. The surface $\Sigma_{h,\beta}$ is uniformly bounded with respect to $0<\beta<1$. Let us consider the limit as $\beta\to 0$ and observe that 
\begin{itemize}
\item[(i)] $\lim_{\beta\downarrow 0} R(h,\beta)=\frac1h,$
\item [(ii)] the radii of the two circular arcs $\gamma_2$ and $\gamma_3$ asymptotically coincide
\item[(iii)] $\gamma_1$ asymptotically shrinks to the origin
\item[(iv)] $\gamma_4$ and $\gamma_5$ shrink to the point $(0,\frac2h)$.
\end{itemize}
Hence we have
\[
\lim_{\beta\downarrow 0} V\left({\Omega_{h,\beta}}\right)=0,
\]
and
\[
\lim_{\beta\downarrow 0} S\left({\Omega_{h,\beta}}\right)=\frac{8\pi}{h^2},
\]
Moreover we have 
\[
-\frac1{R(h,\beta)} \le H\le h, \quad\mbox{for all } \beta>0
\]
and $H$ is uniformly bounded from above and below as long as $0<\beta<1$. 

Once we observe that
\[
\lim_{\beta\downarrow 0} \left\|H\right\|_{L^\infty(\partial\Omega_{h,\beta})}=h,
\]
we can choose $\sqrt{\frac{8\pi}{\epsilon+8\pi}}<h<1$ and $\beta$ small enough to complete the proof.\end{proof}

\begin{proof}[Proof of Corollary \ref{corollario}]
For all $\epsilon>0$ let $\Omega_\epsilon$ be the set defined as in Theorem \ref{main}. We consider then $\displaystyle\widetilde\Omega_\epsilon=\frac{\Omega_\epsilon}{V(\Omega_\epsilon)^{\frac13}}$ rescaled so that the volume is equal to $1$.
The we have
\[
\left\|H\right\|^p_{L^p(\partial\widetilde\Omega_\epsilon)}\le \left\|H\right\|^p_{L^\infty(\partial\widetilde\Omega_\epsilon)} S(\partial\widetilde\Omega_\epsilon)\le (8\pi+\epsilon) \epsilon^{\frac{p-2}{3}},
\]
which implies \eqref{eq_inf}.

Moreover for every $2<p\le\infty$ we deduce
\[
\left\|H\right\|^p_{L^p(\widetilde\Omega^*_\epsilon)}=4\pi \left(\frac{4\pi}{3}\right)^\frac{p-2}{3}\ge 4\pi,
\]
which for $\epsilon$ small enough gives
\[
\left\|H\right\|^p_{L^p(\widetilde\Omega^*_\epsilon)}<\left\|H\right\|^p_{L^p(\widetilde\Omega_\epsilon)}.
\]

\end{proof}

\end{document}